\UseRawInputEncoding
\documentclass[12pt]{article}

\usepackage{dsfont}
\usepackage{amsfonts,amsmath,amsthm}
\usepackage{algorithm, algorithmic}
\usepackage{color}
\usepackage{graphicx}
\usepackage{stmaryrd}        

\usepackage[paper=a4paper,dvips,top=2cm,left=2cm,right=2cm,
    foot=1cm,bottom=4cm]{geometry}

\begin{document}

\title{Dual Quaternion Matrices in Multi-Agent Formation Control}
\author{ Liqun Qi\footnote{
    Department of Mathematics, School of Science, Hangzhou Dianzi University, Hangzhou 310018 China
    ({\tt maqilq@polyu.edu.hk}).}
    \and  \
    Xiangke Wang\thanks{College of Mechatronics and Automation, National University of Defence Technology, Changsha, 410073, China; ({\tt
xkwang@nudt.edu.cn}).  
}
     \and and \
    Ziyan Luo\footnote{Corresponding author, Department of Mathematics,
  Beijing Jiaotong University, Beijing 100044, China. ({\tt zyluo@bjtu.edu.cn}). This author's work was supported by Beijing Natural Science Foundation (Grant No. Z190002).}
}
\date{\today}
\maketitle

\begin{abstract}
Three kinds of dual quaternion matrices associated with the mutual visibility graph, namely the relative configuration adjacency matrix, the logarithm adjacency matrix and the relative twist adjacency matrix,  play important roles in multi-agent formation control.   In this paper, we study their properties and applications.   We show that the relative configuration adjacency matrix and the logarithm adjacency matrix are both Hermitian matrices, and thus have very nice spectral properties.   We introduce dual quaternion Laplacian matrices, and prove a Gershgorin-type theorem for square dual quaternion Hermitian matrices, for studying properties of dual quaternion Laplacian matrices.  The role of the dual quaternion Laplacian matrices in formation control is discussed.

\medskip



\end{abstract}

\renewcommand{\Re}{\mathds{R}}
\newcommand{\rank}{\mathrm{rank}}
\renewcommand{\span}{\mathrm{span}}
\newcommand{\X}{\mathcal{X}}
\newcommand{\A}{\mathcal{A}}
\newcommand{\I}{\mathcal{I}}
\newcommand{\B}{\mathcal{B}}
\newcommand{\C}{\mathcal{C}}
\newcommand{\OO}{\mathcal{O}}
\newcommand{\e}{\mathbf{e}}
\newcommand{\0}{\mathbf{0}}
\newcommand{\dd}{\mathbf{d}}
\newcommand{\ii}{\mathbf{i}}
\newcommand{\jj}{\mathbf{j}}
\newcommand{\kk}{\mathbf{k}}
\newcommand{\va}{\mathbf{a}}
\newcommand{\vb}{\mathbf{b}}
\newcommand{\vc}{\mathbf{c}}
\newcommand{\vg}{\mathbf{g}}
\newcommand{\vr}{\mathbf{r}}
\newcommand{\vt}{\rm{vec}}
\newcommand{\vx}{\mathbf{x}}
\newcommand{\vy}{\mathbf{y}}
\newcommand{\vu}{\mathbf{u}}
\newcommand{\vv}{\mathbf{v}}
\newcommand{\y}{\mathbf{y}}
\newcommand{\vz}{\mathbf{z}}
\newcommand{\T}{\top}

\newtheorem{Thm}{Theorem}[section]
\newtheorem{Def}[Thm]{Definition}
\newtheorem{Ass}[Thm]{Assumption}
\newtheorem{Lem}[Thm]{Lemma}
\newtheorem{Prop}[Thm]{Proposition}
\newtheorem{Cor}[Thm]{Corollary}
\newtheorem{example}[Thm]{Example}
\newtheorem{remark}[Thm]{Remark}

\section{Introduction}
In 1873, Clifford \cite{Cl73} introduced dual quaternions.   Dual quaternions now have wide applications in robotics, 3D motion modelling and control, and computer graphics \cite{ACVL17, BK20, BLH19, CKJC16, Da99, HY03, Ke12, LLB13, PSG19, WYL12}.

In 2011, Wang \cite{Wx11} proposed the study of dual quaternion matrices in his research on formation control in 3D space.   In \cite{WYZ}, Wang and his co-authors studied the application of dual quaternion matrices in the multiple rigid-bodies rendezvous problem \cite{AOSY99}, and proposed three dual quaternion matrices: the relative configuration adjacency matrix, the logarithm adjacency matrix and the relative twist adjacency matrix.   
In \cite{QL22}, Qi and Luo studied the spectral properties of dual quaternion matrices.  If a dual number is a right eigenvalue of a square dual quatrernion matrix, then it is also a left eigenvalue of that matrix.   They call such a right eigenvalue an eigenvalue.    They showed that an $n \times n$ dual quaternion Hermitian matrix has exactly $n$ eigenvalues and no other right eigenvalues.   Such a matrix is positive semi-definite or positive definite if and only if its eigenvalues are all nonnegative or positive, respectively, in the sense of \cite{QLY22}.   Based upon this, a singular value decomposition theorem for general dual quaternion matrices was established.    Dual quaternion matrices were also studied in \cite{LQY22}.

In this paper, we show that the relative configuration adjacency matrix and the logarithm adjacency matrix
are dual quaternion Hermitian matrices, propose a dual quaternion Laplacian matrix theory for the multi-robot rendezvous problem, and study the properties and application in formation control, of dual quaternion Laplacian matrices.

The rest of this paper is distributed as follows.   Preliminary knowledge on dual numbers, quaternions, dual quaternions and dual quaternion matrices, is given in Section 2.  In Section 3, we show that the relative configuration adjacency matrix and the logarithm adjacency matrix
are dual quaternion Hermitian matrices, and the sum of the logarithms of relative configurations over any cycle vanishes to zero.    Dual quaternion Laplacian matrices were introduced in Section 4.  A Gershgorin-type theorem for dual quaternion Hermitian matrices was given there for the discussion of properties of dual quaternion Laplacian matrices.   The role of the dual quaternion Laplacian matrices in formation control is discussed in Section 5.

We denote scalars, vectors and matrices by small letters, bold small letters and capital letters, respectively.   Dual numbers and dual quaternions are distinguished by a hat symbol.

\section{Preliminaries}

\subsection{Dual Numbers}

Denote $\mathbb R$ and $\hat {\mathbb R}$ as the set of the real numbers, and the set of the dual numbers, respectively.  Following the literature such as \cite{WYL12}, we use the hat symbol to denote dual numbers and dual quaternions.   A dual number $\hat q$ has the form $\hat q = q + q_d\epsilon$, where $q$ and $q_d$ are real numbers,  and $\epsilon$ is the infinitesimal unit, satisfying $\epsilon^2 = 0$.   We call $q$ 
the standard part of $\hat q$, and $q_d$ the dual part or the infinitesimal part of $\hat q$.  The infinitesimal unit $\epsilon$ is commutative in multiplication with real numbers, complex numbers and quaternion numbers.  The dual numbers form a commutative algebra of dimension two over the reals. If $q \not = 0$, we say that $\hat q$ is appreciable, otherwise, we say that $\hat q$ is infinitesimal.  Note that a real number is a dual number with a zero dual part.   Then the dual zero is still $0$, and the dual identity is still $1$.

In \cite{QLY22}, a total order was introduced for dual numbers.  Given two dual numbers $\hat p, \hat q \in \hat {\mathbb R}$, $\hat p = p + p_d\epsilon$, $\hat q = q + q_d\epsilon$, where $p$, $p_d$, $q$ and $q_d$ are real numbers, we say that $\hat p \le \hat q$, if either $p < q$, or $p = q$ and $p_d \le q_d$.  In particular, we say that $\hat p$ is positive, nonnegative, nonpositive or negative, if $\hat p > 0$, $\hat p \ge 0$, $\hat p \le 0$ or $\hat p < 0$, respectively.


\subsection{Quaternions}
We adopt the notation $\mathbb Q$ to denote the set of the quaternions. A quaternion $q$ has the form
$q = q_0 + q_1\ii + q_2\jj + q_3\kk,$
where $q_0, q_1, q_2$ and $q_3$ are real numbers, $\ii, \jj$ and $\kk$ are three imaginary units of quaternions, satisfying
$$\ii^2 = \jj^2 = \kk^2 =\ii\jj\kk = -1, ~~\ii\jj = -\jj\ii = \kk, ~~ \jj\kk = - \kk\jj = \ii, ~~\kk\ii = -\ii\kk = \jj.$$
The real part of $q$ is Re$(q) = q_0$.   The imaginary part of $q$ is Im$(q) = q_1\ii + q_2\jj +q_3\kk$.
The multiplication of quaternions satisfies the distribution law, but is noncommutative.

The conjugate of $q = q_0 + q_1\ii + q_2\jj + q_3\kk$ is
$q^* := q_0 - q_1\ii - q_2\jj - q_3\kk.$
The magnitude of $q$ is
$|q| = \sqrt{q_0^2+q_1^2+q_2^2+q_3^2}.$
It follows that the inverse of a nonzero quaternion $q$ is given by
$q^{-1} = {q^* / |q|^2}.$ For any two quaternions $p$ and $q$, we have $(pq)^* = q^*p^*$.

A quaternion is called imaginary if its real part is zero.  If $q$ is imaginary, then $q^* = -q$.  In the engineering literature \cite{WYL12}, it is called a vector quaternion.   Various 3D vectors, such as position vectors, displacement vectors, linear velocity vectors, and angular velocity vectors, can be represented as imaginary quaternions.   As 3D vectors have cross product, imaginary quaternions can also have cross products.  That is, the cross product of two imaginary quaternions is an imaginary quaterion, with its imaginary part is the cross product of the two original imaginary quaternions.

If $|q|=1$, then $q$ is called a unit quaternion, or a rotation quaternion.  A spatial rotation around a fixed point of $\theta$  radians about a unit axis $(x_1,x_2,x_3)$ that denotes the Euler axis is given by the unit quaternion $q=(\cos(\theta /2),x_1\sin(\theta /2),x_2\sin(\theta /2),x_3\sin(\theta /2))$, i.e.,
\begin{equation} \label{uq}
q = \cos(\theta /2) + x_1\sin(\theta /2)\ii + x_2\sin(\theta /2)\jj + x_3\sin(\theta /2)\kk = e^{{\theta \over 2}x}.
\end{equation}
where the unit axis $x$ is an imaginary unit quaternion $x = x_1\ii + x_2\jj + x_3\kk$.   Thus, we may also write
\begin{equation} \label{lnuq}
\ln q = {\theta \over 2}x.
\end{equation}
Note that a unit quaternion $q$ is always invertible and $q^{-1} = q^*$. Furthermore, since in general
it is possible that $pq \not = qp$, it is also possible $\ln pq \not = \ln p + \ln q$.

Two quaternions $p$ and $q$ are said to be similar if there is a nonzero quaternion $u$ such that $p = u^{-1}qu$.   We denote $p \sim q$.    It is easy to check that $\sim$ is an equivalence relation on the quaternions.   Denote by $[q]$ the equivalence class containing $q$. Then $[q]$ is a singleton if and only if $q$ is a real number.

Denote the collection of $n$-dimensional quaternion  vectors by ${\mathbb {Q}}^n$. For $\vx = (x_1, x_2,\cdots, x_n)^\top, \vy = (y_1, y_2,\cdots, y_n)^\top  \in {\mathbb {Q}}^n$, define $\vx^*\vy = \sum_{i=1}^n x_i^*y_i$, where $\vx^* = (x_1^*, x_2^*,\cdots, x_n^*)$ is the conjugate transpose of $\vx$.

\subsection{Dual Quaternions}

Denote the set of dual quaternions by $\hat {\mathbb Q}$.   A dual quaternion $\hat q \in \hat {\mathbb Q}$ has the form
\begin{equation} \label{dq}
\hat q = q + q_d\epsilon,
\end{equation}
where $q, q_d \in \mathbb {Q}$ are the standard part and the dual part of $\hat q$, respectively. If $q \not = 0$, then we say that $\hat q$ is appreciable.  If $q$ and $q_d$ are imaginary quaternions, then $\hat q$ is called an imaginary dual quaternion. Here, the infinitesimal unit $\epsilon$ is commutative with each of those three imaginary units of quaternions. Thus, $q_d\epsilon= \epsilon q_d$ for any $q_d \in \mathbb {Q}$.


The conjugate of $\hat q$ is
\begin{equation} \label{conjugate}
\hat q^* = q^* + q_d^*\epsilon.
\end{equation}
By this, if $\hat q = \hat q^*$, then $\hat q$ is a dual number.  If $\hat q$ is imaginary, then $\hat q^* = - \hat q$.

The magnitude of $\hat q$ was defined in \cite{QLY22} as
\begin{equation} \label{magnitude}
|\hat q| := \left\{ \begin{aligned} |q| + {(qq_d^*+q_d q^*) \over 2|q|}\epsilon, & \ {\rm if}\  q \not = 0, \\
|q_d|\epsilon, &  \ {\rm otherwise},
\end{aligned} \right.
\end{equation}
which is a dual number.

Suppose that we have two dual quaternions $\hat p = p + p_d \epsilon$ and $\hat q = q + q_d \epsilon$, their addition and multiplications are defined as
$$\hat p+\hat q = \left(p + q\right) + \left(p_d + q_d\right)\epsilon$$
and
$$\hat p\hat q = pq + \left(pq_d + p_d q\right)\epsilon.$$
See \cite{Ke12, LLB13}. Note that under these arithmetic rules, a dual number is commutative with a dual quaternion or a dual quaternion vector.

A dual quaternion $\hat q$ is called invertible if there exists a quaternion $\hat p$ such that $\hat p\hat q = \hat q\hat p =1$.  We can derive that $\hat q$ is invertible if and only if  $\hat q$ is appreciable. In this case, we have
$$\hat q^{-1} = q^{-1} - q^{-1}q_d q^{-1} \epsilon.$$


If $|\hat q| = 1$, then $\hat q$ is called a unit dual quaternion. A unit dual quaternion $\hat q$ is always invertible and we have ${\hat q}^{-1} = {\hat q}^*$.   The 3D motion of a rigid body can be represented by a unit dual quaternion.
We have
$$\hat q\hat q^* = (q + q_d\epsilon)(q^* + q_d^*\epsilon)= qq^* + (qq_d^* + q_d q^*)\epsilon = \hat q^*\hat q.$$
Thus, $\hat q$ is a unit dual quaternion if and only if $q$ is a unit quaternion, and
\begin{equation} \label{udq1}
qq_d^* + q_d q^* = q^*q_d + q_d^* q=0.
\end{equation}
For example, $\hat q = \ii + \jj \epsilon$ is a unit dual quaternion. Suppose that there is a rotation $q \in {\mathbb Q}$ succeeded by a translation $p^b \in {\mathbb Q}$, where $p^b$ is an imaginary quaternion.   Here, following \cite{WYL12}, we use superscripts $b$ and $s$ to represent the relation of the rigid body motion with respect to the body frame attached to the rigid body and the spatial frame which is relative to a fixed coordinate frame.
Then the whole transformation can be represented using unit dual quaternion $\hat q = q + q_d \epsilon$, where $q_d = {1 \over 2}qp^b$.   Note that we have
$$qq_d^* + q_d q^* = {1 \over 2}\left[q(p^b)^*q^* + qp^bq^*\right] = {1 \over 2}q\left[(p^b)^*+p^b\right]q^* = 0.$$
Thus, a transformation of a rigid body can be represented by a unit dual quaternion
\begin{equation} \label{udq}
\hat q = q + {\epsilon \over 2}qp^b,
\end{equation}
where $q$ is a unit quaternion to represent the rotation, and $p^b$ is the imaginary quaternion to represent the translation or the position.   Here, unit quaternion $q$ serves as a rotation, taking coordinates  $r^o$ of a point in the original frame to coordinates $r^n$ in the new frame by
\begin{equation} \label{coordinates}
r^n = q^*r^oq,
\end{equation}
where $r^o$ and $r^n$ are two imaginary quaternions, their superscripts $o$ and $n$ represent ``original'' and ``new'' respectively.    On the other hand, every attitude of a rigid body which is free to rotate relative to a fixed frame can be identified by a unique unit quaternion $q$.    Thus, in (\ref{udq}), $q$ is the attitude of the rigid body, while $\hat q$ represents the transformation.

The configuration change rate of a rigid body can be expressed by
\begin{equation} \label{kinematics}
\dot {\hat q} \equiv {d {\hat q} \over dt} = {1 \over 2}\hat q \hat \xi^b,
\end{equation}
where the unit dual quaternion $\hat q$
is the configuration of that rigid body, expressed by (\ref{udq}), and the imaginary dual quaternion
\begin{equation} \label{twist}
\hat \xi^b = \omega^b + \epsilon(\dot p^b + \omega^b \times p^b)
\end{equation}
is the twist of rigid body, with angular velocity $\omega^b$ and linear velocity $\dot p^b \equiv {d p^b \over dt}$.

Combining (\ref{udq}) with (\ref{lnuq}), we have
\begin{equation} \label{lnudq}
\ln \hat q = {1 \over 2}(\theta x + \epsilon p^b).
\end{equation}

Given an imaginary quaternion $v$ and a unit quaternion $q$, the adjoint transformation is defined as
\begin{equation} \label{adjoint}
{\rm Ad}_q v = qvq^*.
\end{equation}
Similarly, given an imaginary dual quaternion $\hat v$ and a unit dual quaternion $\hat q$, the adjoint transformation of the unit dual quaternion $\hat q$ is defined as
\begin{equation} \label{adjointudq}
{\rm Ad}_{\hat q} \hat v = \hat q\hat v\hat q^*.
\end{equation}

A unit dual quaternion $\hat q$ serves as both a specification of the configuration of a rigid body and a transformation taking the coordinates of a point from one frame to another via rotation and translation.
In (\ref{udq}), if $\hat q$ is the configuration of the rigid body, then $q$ and $p^b$ are the attitude of and position of the rigid body respectively.
They have the following property:

Property A: If the configurations of rigid bodies $i$ and $j$ are $\hat q_i$ and $\hat q_j$, and the transformation (relative configuration) from rigid body $i$ to $j$ is $\hat q_{ij}$, then $\hat q_j = \hat q_i\hat q_{ij}$.

Consider two unit dual quaternions $\hat q$ and $\hat q_t$ defined as in (\ref{udq}).  The left-invariant error from $\hat q$ to $\hat q_t$ is
\begin{equation}
\hat q_e = \hat q_t^*\hat q = q_e + {\epsilon \over 2}q_ep_e^b,
\end{equation}
where $q_e = q_t^*q$ and $p_e^b = p^b - {\rm Ad}_{q_e^*} p_t^b$.

Two dual quaternions $\hat p$ and $\hat q$ are said to be similar if there is an appreciable quaternion $\hat u$ such that $\hat p = \hat u^{-1}\hat q\hat u$.   We denote $\hat p \sim \hat q$.    Then $\sim$ is an equivalence relation on the dual quaternions.   Denote by $[\hat q]$ the equivalence class containing $\hat q$. Then $[\hat q]$ is a singleton if and only if $\hat q$ is a dual number.

Denote the collection of $n$-dimensional dual quaternion  vectors by ${\hat {\mathbb Q}}^n$.

For $\hat \vx = (\hat x_1, \hat x_2,\cdots, \hat x_n)^\top$, $\hat \vy = (\hat y_1, \hat y_2,\cdots, \hat y_n)^\top  \in {\hat {\mathbb Q}}^n$, define $\hat \vx^*\hat \vy = \sum_{i=1}^n \hat x_i^*\hat y_i$, where $\hat \vx^* = (\hat x_1^*, \hat x_2^*,\cdots, \hat x_n^*)$ is the conjugate transpose of $\hat \vx$.   We say $\hat \vx$ is appreciable if at least one of its component is appreciable.  We also say that $\hat \vx$ and $\hat \vy$ are orthogonal to each other if $\hat \vx^*\hat \vy = 0$.  By \cite{QLY22}, for any $\hat \vx \in {\hat {\mathbb Q}}^n$, $\hat \vx^*\hat \vx$ is a nonnegative dual number, and if $\hat \vx$ is appreciable, $\hat \vx^*\hat \vx$ is a positive dual number.


\subsection{Dual Quaternion Matrices}

Denote the collections of $m \times n$ dual quaternion  matrices by ${\hat {\mathbb Q}}^{m \times n}$.  Then $\hat A \in {\hat {\mathbb Q}}^{m \times n}$ can be written as
$$\hat A = A + A_d \epsilon,$$
where $A, A_d \in {\mathbb {Q}}^{m \times n}$ are the standard part and the infinitesimal part of $\hat A$, respectively.

The conjugate transpose of $\hat A$ is
$$\hat A^* = A^* + A_d^* \epsilon.$$

A square dual quaternion matrix $\hat A \in {\hat {\mathbb Q}}^{n \times n}$ is called normal if $\hat A^*\hat A = \hat A\hat A^*$, Hermitian if $\hat A^* = \hat A$; unitary if $\hat A^*\hat A = I_n$, where $I_n$ is the $n \times n$ identity matrix; and invertible (nonsingular) if $\hat A\hat B = \hat B\hat A = I_n$ for some $\hat B \in {\hat {\mathbb Q}}^{n \times n}$.  Indeed, if $\hat A= A+A_d \epsilon$ and $\hat B= B+B_d \epsilon$ satisfy $\hat A\hat B = I_n$, then $B = A^{-1}$ and $B_d = - A^{-1}A_d  A^{-1}$. This further implies $\hat B\hat A = I_n$.   Thus, the inverse of $\hat A$ is unique, denoted by $\hat A^{-1}$, and we have $\hat A^{-1} =  A^{-1} - A^{-1}A_d  A^{-1}\epsilon$.

We have $(\hat A\hat B)^{-1} = \hat B^{-1}\hat A^{-1}$ if $\hat A$ and $\hat B$ are invertible, and $\left(\hat A^*\right)^{-1} = \left(\hat A^{-1}\right)^*$ if $\hat A$ is invertible.

Suppose that $\hat A \in {\hat {\mathbb Q}}^{n \times n}$ is an Hermitian matrix.  For any $\hat \vx \in {\hat {\mathbb Q}}^n$, we have $(\hat \vx^*\hat A\hat \vx)^* = \hat \vx^*\hat A\hat \vx.$
Thus, $\hat \vx^*\hat A\hat \vx$ is a dual number.  A dual quaternion Hermitian matrix $\hat A \in {\hat {\mathbb Q}}^{n \times n}$ is called positive semi-definite if for any $\hat \vx \in {\hat {\mathbb Q}}^n$, $\hat \vx^*\hat A\hat \vx \ge 0$; $\hat A$ is called positive definite if for any $\hat \vx \in {\hat {\mathbb Q}}^n$ with $\hat \vx$ being appreciable,  we have $\hat \vx^*\hat A\hat \vx > 0$ and is appreciable.
A square dual quaternion matrix $\hat A \in {\hat {\mathbb Q}}^{n \times n}$ is unitary if and only if its column (row) vectors form an orthonormal basis of ${\hat {\mathbb Q}}^n$.
Then, $\hat A$ is a dual quaternion Hermitian matrix if and only if $A$ and $A_d$ are two quaternion Hermitian matrices.

Suppose that $\hat A \in {\hat {\mathbb Q}}^{n \times n}$.  If there are $\hat \lambda \in \hat {\mathbb Q}$, $\hat \vx \in {\hat {\mathbb Q}}^n$, where $\hat \vx$ is appreciable, such that
\begin{equation}  \label{eigen}
\hat A\hat \vx = \hat \vx\hat \lambda,
\end{equation}
then we say that $\hat \lambda$ is a right eigenvalue of $\hat A$, with $\hat \vx$ as an associated right eigenvector.
If there are $\hat \lambda \in \hat {\mathbb Q}$, $\hat \vx \in {\hat {\mathbb Q}}^n$, where $\hat \vx$ is appreciable, such that
$$\hat A\hat \vx = \hat \lambda\hat \vx,$$
then we say that $\hat \lambda$ is a left eigenvalue of $\hat A$, with $\hat \vx$ as an associated left eigenvector.   If $\hat \lambda$ is a dual number and a right eigenvalue of $\hat A$, then it is also a left eigenvalue of $\hat A$, as  a dual number is commutative with a dual quaternion vector.   In this case, we say that it is an eigenvalue of $\hat A$, with $\hat \vx$ as an associated eigenvector.

The following theorems are proved in \cite{QL22}.

\begin{Thm} \label{p3.0}
Suppose that $\hat \lambda = \lambda + \lambda_d\epsilon \in \hat {\mathbb Q}$ is a right eigenvalue of $\hat A \in {\hat {\mathbb Q}}^{n \times n}$, with associated right eigenvector $\hat \vx = \vx + \vx_d\epsilon \in {\hat {\mathbb Q}}^n$. Then
\begin{equation} \label{e5}
\hat \lambda = {\hat \vx^* \hat A\hat \vx \over \hat \vx^*\hat \vx},
\end{equation}
$\lambda$ is a right eigenvalue of the quaternion matrix $A$ with a
right eigenvector $\vx$, i.e., $\vx \not = \0$ and
\begin{equation} \label{e3}
A\vx = \vx\lambda.
\end{equation}
We also have
\begin{equation} \label{e6}
\lambda = {\vx^* A\vx \over \vx^*\vx}.
\end{equation}
\end{Thm}

\begin{Thm} \label{t3.2}
A right eigenvalue $\hat \lambda$ of an Hermitian matrix $\hat A = A + A_d\epsilon \in {\hat {\mathbb Q}}^{n \times n}$ must be a dual number, hence an eigenvalue of $\hat A$, and its standard part $\lambda$ is a right eigenvalue of the quaternion Hermitian matrix $A$.   Furthermore, assume that $\hat \lambda = \lambda + \lambda_d \epsilon$,
$\hat \vx = \vx + \vx_d \epsilon \in {\hat {\mathbb Q}}^n$ is an eigenvector of $\hat A$, associate with the eigenvalue $\hat \lambda$,
where $\vx, \vx_d \in {\mathbb Q}^n$.   Then we have
\begin{equation} \label{e7}
\lambda_d = {\vx^* A_d \vx \over \vx^*\vx}.
\end{equation}

A dual quaternion Hermitian matrix has at most $n$ dual number eigenvalues and no other right eigenvalues.

An eigenvalue of a positive semi-definite Hermitian matrix $\hat A \in {\hat {\mathbb Q}}^{n \times n}$ must be a nonnegative dual number.   In that case, $A$ must be positive semi-definite.
An eigenvalue of a positive definite Hermitian matrix $\hat A \in {\hat {\mathbb Q}}^{n \times n}$ must be an appreciable positive dual number.   In that case, $A$ must be positive definite.
\end{Thm}

\begin{Thm} \label{t4.1}
Suppose that $\hat A=A+A_d \epsilon \in {\hat {\mathbb Q}}^{n \times n}$ is an Hermitian matrix.  Then there are unitary matrix $\hat U \in {\hat {\mathbb Q}}^{n \times n}$ and a diagonal matrix $\hat \Sigma \in {\hat {\mathbb R}}^{n \times n}$ such that $\hat \Sigma = \hat U^*\hat A\hat U$, where
\begin{equation} \label{eee1}
\hat \Sigma:= {\rm diag}\left(\lambda_1+\lambda_{1,1}\epsilon,\cdots, \lambda_1+\lambda_{1,k_1}\epsilon, \lambda_2+\lambda_{2,1}\epsilon,\cdots, \lambda_r+\lambda_{r,k_r}\epsilon\right),
\end{equation}
with the diagonal entries of $\hat \Sigma$ being $n$  eigenvalues of $\hat A$,
\begin{equation} \label{eee2}
\hat A\hat \vu_{i, j} = \hat \vu_{i, j}(\lambda_i +\lambda_{i, j}\epsilon),
\end{equation}
for $j = 1, \cdots, k_i$ and $i = 1, \cdots, r$, $\hat U = (\hat \vu_{1,1}, \cdots, \hat \vu_{1, k_1}, \cdots, \hat \vu_{r, k_r})$,
$\lambda_1 > \lambda_2 > \cdots > \lambda_r$ are real numbers, $\lambda_i$ is a $k_i$-multiple right eigenvalue of $A$, $\lambda_{i, 1} \ge \lambda_{i, 2} \ge \cdots \ge \lambda_{i, k_i}$ are also real numbers, $\sum_{i=1}^r k_i = n$.   Counting possible multiplicities $\lambda_{i, j}$, the form $\hat \Sigma$ is unique.

\end{Thm}

\begin{Thm} \label{t2.4}
Suppose that $\hat A \in {\hat {\mathbb Q}}^{n \times n}$ is Hermitian.  Then $\hat A$ has exactly $n$ eigenvalues, which are all dual numbers.  There are also $n$ eigenvectors, associated with these $n$ eigenvalues.   The Hermitian matrix $\hat A$ is positive semi-definite or definite if and only if all of these eigenvalues are nonnegative, or positive and appreciable, respectively.
\end{Thm}

\section{Graph Model of Multi-Agent Formation Control}

Following \cite{LWHF14, WYZ}, we model the rigid bodies in the multi-agent formation control by an undirected graph.   We assume that the rigid-bodies are omni-directional, i.e., rigid-body $i$ can sense rigid body $j$ if and only if rigid-body $j$ can sense rigid body $i$.  Then multiple rigid-bodies are modelled by an undirected graph $G = (V, E)$, called the mutual visibility graph.  Here, the vertex set $V$ is the set of rigid bodies.   Edge $(i, j) \in E$ if rigid-bodies $i$ and $j$ are mutual visual.  Let $n = |V|$, where $|V|$ means the cardinality of the set $V$.  As in spectral graph theory, we have adjacency matrix $\hat A(G)$.   Then we have the following three $n \times n$ dual quaternion matrices: the relative configuration adjacency matrix $\hat C(G)$, the logarithm adjacency matrix $\hat Ln(G)$, and the relative twist adjacency matrix $\hat T(G)$.  Recall the concepts of relative configuration, twist and logarithm of dual quaternions discussed in Section 2.  The $(i, j)$ entries of $\hat C(G), \hat Ln(G)$ and $\hat T(G)$ are the relative configuration $\hat q_{ij}$, the logarithm of the relative configuration $\ln \hat q_{ij} = {1 \over 2}(\theta_{ij} x_{ij} + \epsilon p_{ij}^b)$, and the relative twist $\hat \xi_{ij}^b$, respectively, if $i, j = 1, \cdots, n, i \not = j$ and $(i, j) \in E$.   The $(i, j)$ entries of $\hat C(G), \hat Ln(G)$ and $\hat T(G)$ are zero otherwise.

Note that $\hat q_{ij}$ is a unit dual quaternion, $\ln \hat q_{ij}$ is a unit imaginary dual quaternion, and $\hat \xi_{ij}^b$ is an imaginary dual quaternion.

We have the following theorem.

\begin{Thm} \label{dqmhermitian}
The relative configuration adjacency matrix $\hat C(G)$ and the logarithm adjacency matrix $\hat Ln(G)$ are dual quaternion Hermitian matrices.    For a cycle $\{ j_1, \cdots, j_k \}$,  with $j_{k+1} = j_1$, of $G$, we have
\begin{equation} \label{cycle}
\prod_{i=1}^k \hat q_{j_ij_{i+1}} = 1.
\end{equation}
\end{Thm}
\begin{proof}
For $i, j = 1, \cdots, n$, $i \not = j$, by Property A,  $\hat q_j = \hat q_i \hat q_{ij}$.   Thus, we have
$$\hat q_{ij} = (\hat q_i)^{-1} \hat q_j.$$
Thus,
\begin{equation} \label{cginverse}
\hat q_{ji} = (\hat q_j)^{-1}\hat q_i = (\hat q_{ij})^{-1}.
\end{equation}
Since $\hat q_{ij}$ is a unit dual quaternion, this implies that
\begin{equation} \label{entryconj}
\hat q_{ji} = (\hat q_{ij})^*.
\end{equation}
This shows that $\hat C(G)$ is an Hermitian matrix.  By (\ref{cginverse}), we have
$$\ln \hat q_{ji} = - \ln \hat q_{ij}.$$
Since $\ln \hat q_{ij}$ is an imaginary dual quaternion, this implies
$$\ln \hat q_{ji} = \left(\ln \hat q_{ij}\right)^*.$$
Thus, $\hat Ln(G)$ is also an Hermitian matrix.

On the cycle $\{ j_1, \cdots, j_k \}$, we have
$\hat q_{j_{i+1}} = \hat q_{j_i} \hat q_{j_ij_{i+1}}$ for $i = 1, \cdots, k$.   This implies (\ref{cycle}).   
\end{proof}

Formula (\ref{cycle}) shows that if there is a cycle in the mutual visibility graph $G$, then we may always delete one edge in that cycle without affecting the description of mutual visibility.   Repeating this process, eventually the mutual visibility graph contains no cycle.   If the original mutual visibility graph is connected, then we obtain a reduced mutual visibility graph, which is a tree.  Then we may apply the formation control strategy described in \cite{WYL12}.

\section{Dual Quaternion Laplacian Matrices}

In the last section, we show that the relative configuration adjacency matrix $\hat C(G)$ and the logarithm adjacency matrix $\hat Ln(G)$ are dual quaternion Hermitian matrices.  This lays the base for us to study the stability issue of the multi-agent formation control problem.   In this section, we set the dual quaternion Laplacian matrix theory such that we have a tool to tackle this issue.  Before doing this, we
establish a Gershgorian type theorem for dual quaternion Hermitian matrices.

\begin{Thm} \label{Gershgorin}
Suppose that $\hat H = (\hat h_{ij}) \in {\hat {\mathbb Q}}^{n\times n}$ is an Hermitian matrix.  Then the $n$ eigenvalues of $\hat H$ lay in the following $n$ sets:
$$S_i = \left\{ \hat \lambda \in {\hat {\mathbb D}} : | \lambda - \hat h_{ii} | \le \sum_{j \not = i} |\hat h_{ij}| \right\},$$
for $i = 1, \cdots, n$.
\end{Thm}
\begin{proof}   Suppose that $\hat \lambda$ is an eigenvalue of $H$ with an eigenvector $\hat \vx$.   Then $\hat \lambda \in {\hat {\mathbb R}}$, and $\hat \vx \in {\hat {\mathbb Q}}^n$ is appreciable.  By (\ref{eigen}), we have
$$\sum_{j=1}^n \hat h_{ij}\hat x_j = \hat \lambda \hat x_i,$$
i.e.,
$$\sum_{j\not = i} \hat h_{ij}\hat x_j = \left(\hat \lambda - \hat h_{ii} \right)\hat x_i,$$
for $i = 1, \cdots, n$.   Assume that $\left|\hat x_i\right| = \max_{j = 1, \cdots, n} \left|\hat x_j\right|$.   Since $\hat \vx$ is appreciable, $\left|\hat x_i\right|$ must be appreciable.   Then we have
$$| \lambda - \hat h_{ii} | \le \sum_{j \not = i} |\hat h_{ij}|.$$
This proves the theorem.
\end{proof}

 Suppose that we have a mutual visibility graph $G = (V, E)$ with $|V| = n$.  For each $i \in V$, let $d(i)$ be the number of edges $(i, j) \in E$.  Let $D(G) = (d_{ij})$ be an $n \times n$ real diagonal matrix, with its $i$th diagonal entry being $d(i)$.  Let $A(G) = (a_{ij})$ be an $n \times n$ traceless matrix, with $a_{ij} = 1$ if and only if $(i, j) \in E$.   Here, the word ``traceless'' means that the diagonal entries of $A$ are zero.  Let $L(G) = D(G) - A(G)$.   Then $D(G)$, $A(G)$ and $L(G)$ are the degree matrix, the adjacency matrix and the Laplacian matrix of $G$.  Spectral graph theory are based upon these three matrices.

 We now extend these to dual quaternion matrices. We keep the degree matrix without any change as the degree $d(i)$, which is the number of rigid-bodies that can be sensed by the rigid-body $i$. This is meaningful in formation control.  Then we call an  $n \times n$ dual quaternion Hermitian matrix $\hat H =\left(\hat{h}_{ij}\right)$ a dual quaternion adjacency matrix, if $\hat{h}_{ij} = \left(\hat{h}_{ji}\right)^*$ is a unit dual quaternion for any $i<j$ satisfying that the rigid body $i$ can be sensed by the rigid body $j$, and other entries $0$. Apparently, the relative configuration adjacency matrix $\hat C(G)$ and the logarithm adjacency matrix $\hat Ln(G)$ are of such a type. Furthermore, we call $\hat L(\hat H) = D-\hat H$ the dual quaternion Laplacian matrix, associated with $\hat H$.


 \begin{Thm}  \label{PSD}
 A dual quaternion Laplacian matrix is a positive semi-definite dual quaternion Hermitian matrix.  Its $n$ eigenvalues are nonnegative dual numbers.
 \end{Thm}
 \begin{proof} By the definition of dual quaternion Laplacian matrices and Theorem \ref{Gershgorin}, all the eigenvalues of a dual quaternion Laplacian matrix are nonnegative dual numbers.   By Theorem \ref{t2.4}, this matrix is positive semi-definite.
 \end{proof}

 \section{Formation Control of Multiple Rigid-Bodies}

The objective of this study is to design a universal control scheme without requiring decoupling rotational and translational dynamics to make all the connected rigid-bodies rendezvous, i.e., converge into a predefined or unspecified configuration based upon error dynamics with perfect and instantaneous measurements.   Assume that there are $n$ rigid bodies as described in the last section.   The positions of these $n$ rigid bodies are denoted by a dual quaternion vector $\hat \vz = (\hat z_1, \cdots, \hat z_n)^\top \in {\hat {\mathbb Q}^n}$.  Each rigid body $i$ is assumed to have an onboard sensor to measure relative positions of neighboring rigid bodies, that is $\hat z_{ij} \in \hat {\mathbb Q}$ when $(i, j) \in E$ for the mutual visibility graph $G = (V, E)$.   Here, $\hat z_{i,j}$ may be the relative configuration $\hat q_{ij}$, the logarithm of the relative configuration $\ln \hat q_{ij}$, the relative twist $\hat \xi_{ij}^b$, or simply
$$\hat z_{ij} = \hat z_j - \hat z_i.$$

Assume that each rigid body $i$ has a point kinematic model given by the single integrator
\begin{equation}
{{d \hat z_i} \over dt} = \hat u_i,
\end{equation}
where $\hat u_i \in \mathbb Q$ represents the velocity control input.   Let $\hat \vv = (\hat v_1, \cdots, \hat v_n)^\top \in {\hat {\mathbb Q}^n}$ be the targeted formation configuration, satisfying $\hat L \hat \vv = 0$.    Then we consider the
control law
\begin{equation} \label{control}
\hat u_i= \sum_{(i, j) \in E} \hat l_{ij}(\hat z_j - \hat z_i),
\end{equation}
where $\hat l_{ij}$ are the entries of the Laplacian matrix $\hat L$, for $i = 1, \cdots, n$.

Under the control law (\ref{control}),
the overall closed-loop dynamics of the $n$ rigid bodies becomes
\begin{equation}
{{d \hat \vz} \over dt} = -\hat L\hat  \vz.
\end{equation}

Then, as studied in \cite{LWHF14}, if $\hat L$  is 
positive semi-definite, and its eigenvalues are positive except at one zero eigenvalue, the system is stable.  Hence, Theorem \ref{PSD} paves a way for such a study.

\bigskip




\end{document}